\documentclass[12pt,reqno]{amsart}

  \usepackage{graphicx}
  \usepackage{amsmath}
   \usepackage{amssymb}

\usepackage{amsfonts}
\usepackage{fullpage}
\usepackage{hyperref}
\usepackage{enumerate}

\newtheorem{theorem}{Theorem}
\newtheorem{corollary}[theorem]{Corollary}
\newtheorem{lemma}[theorem]{Lemma}

\newtheorem{definition}[theorem]{Definition}

\title{A Digital Binomial Theorem for Sheffer Sequences}

\author{Toufik Mansour}
\author{Hieu D. Nguyen}

    \address{Department of Mathematics, University of Haifa, 31905 Haifa, Israel}
    \email{toufik@math.haifa.ac.il}
        
    \address{Department of Mathematics, Rowan University, Glassboro, NJ, USA}
    \email{nguyen@rowan.edu}

\date{10-28-2015}

\begin{document}

\subjclass[2010]{Primary 11B65; Secondary 28A80 }
\keywords{binomial theorem, Sierpi\'nski matrix, sum-of-digits, Sheffer polynomial sequence}

\begin{abstract}

We extend the digital binomial theorem to Sheffer polynomial sequences by demonstrating that their corresponding Sierpi\'nski matrices satisfy a multiplication property that is equivalent to the convolution identity for Sheffer sequences.

\end{abstract}

\maketitle

\section{Introduction}
\label{sec:1}

The binomial theorem is a fundamental result in mathematics:
\begin{equation} \label{eq:binomial-theorem}
(x+y)^n=\sum_{k=0}^n \binom{n}{k}x^k y^{n-k}.
\end{equation}
One generalization of the binomial theorem, due to Callan \cite{C} (see also \cite{N1}), expresses the exponents appearing in (\ref{eq:binomial-theorem}) in terms of the binary sum-of-digits function $s(m)$:
\begin{equation} \label{eq:digital-binomial-theorem-carry-free}
(x+y)^{s(n)}  = \sum_{\substack{0\leq m \leq n \\  (m,n-m) \ 
\text{\rm carry-free}}}x^{s(m)}y^{s(n-m)},
\end{equation}
where a pair of non-negative integers $(j,k)$ is said to be carry-free if their sum in binary involves no carries.  We refer to (\ref{eq:digital-binomial-theorem-carry-free}) as the digital binomial theorem.

Several extensions of the digital binomial theorem have recently been found.  For example, a non-binary version is given in \cite{N2} for any integer base $b>2$:
\begin{equation} \label{eq:digital-binomial-theorem-non-binary}
\prod_{i=0}^{N-1}\binom{x+y+n_i-1}{n_i} =\sum_{0\leq m\preceq_b n}\left( \prod_{i=0}^{N-1}\binom{x+m_i-1}{m_i} \prod_{i=0}^{N-1}\binom{y+n_i-m_i-1}{n_i-m_i} \right).
\end{equation}
Here, $n_i$ and $m_i$ denote the $i$-th digit in the base-$b$ expansion of $n$ and $m$, respectively.  Also, $m\preceq_b n$ denotes the fact that $m_i\leq n_i$ for all $i$.   Another example is a $q$-analog given in \cite{MN}:
\begin{equation} \label{eq:q-digital-binomial-theorem}
 \prod_{i=0}^{N-1}\binom{x+q^i y+ n_i-1}{n_i} = \sum_{\substack{0\leq m \leq n \\  (m,n-m) \
\text{\rm carry-free}}}q^{z_{n}(m)}x^{s(m)}y^{s(n-m)}.
\end{equation}

 In this paper, we present a digital binomial theorem for Sheffer sequences and those of binomial type by considering a polynomial generalization of (\ref{eq:binomial-theorem}).  Let $f(t)$ be a delta series and $g(t)$ be an invertible series, i.e.,
 \begin{align*}
 f(t) & =\sum_{k=0}^{\infty}a_kt^k \qquad (a_0=0,a_1\neq 0), \\
 g(t) & = \sum_{k=0}^{\infty} b_kt^k \qquad (b_0\neq 0).
 \end{align*}
The polynomial sequence $s_n(x)$, $n=0,1,\dots$, is said to be Sheffer for $(g(t),f(t))$ if it has generating function  (\cite[Theorem 2.3.4, p. 18]{R})
\begin{equation} \label{eq:sheffer}
\frac{1}{g(\bar{f}(t))}e^{x\bar{f}(t)} = \sum_{n=0}^{\infty} \frac{s_n(x)}{n!}t^n
\end{equation}
where $\bar{f}(t)$ is the compositional inverse of $f(t)$.
The Sheffer sequence $s_n(x)$ is known to satisfy the convolution identity (\cite[Theorem 2.3.9, p. 21]{R})
\begin{equation}\label{eq:sheffer-sequence}
s_n(x+y)=\sum_{k=0}^n \binom{n}{k}p_k(x)s_{n-k}(y).
\end{equation}
Here, $p_n(x)$ is the polynomial sequence associated to $s_n(x)$, i.e., $p_n(x)$ has generating function
\[
e^{x\bar{f}(t)} = \sum_{n=0}^{\infty} \frac{p_n(x)}{n!}t^n.
\]  
It is known that $p_n(x)$ is of binomial type, i.e., $p_n(x)$ satisfies the binomial identity ((\cite[Theorem 2.4.7, p. 26]{R})
\begin{equation}\label{eq:polynomial-binomial-type}
p_n(x+y)=\sum_{k=0}^n \binom{n}{k}p_k(x)p_{n-k}(y).
\end{equation}

In the special case where $\bar{f}(t)=t$ in (\ref{eq:sheffer}) so that $p_n(x)=x^n$, the Sheffer sequence $s_n(x)$ is then called an Appell sequence.  There are many well-known examples of Sheffer sequences, e.g.,
the Bernoulli, Hermite (probabilistic version), and Laguerre polynomials defined by the generating functions
\[
\frac{t}{e^t-1}\exp(xt) = \sum_{n=0}^{\infty} \frac{B_n(x)}{n!}t^n,
\]
\[
\exp(xt-t^2/2) = \sum_{n=0}^{\infty} \frac{H_n(x)}{n!}t^n,
\]
and
\[
\frac{1}{(1-t)^{\alpha+1}}\exp\left(-\frac{xt}{1-t}\right) = \sum_{n=0}^{\infty} \frac{L_n^{\alpha}(x)}{n!}t^n,
\]
respectively.  Observe that $B_n(x)$ and $H_n(x)$ are both Appell sequences.  It is well known that all three Sheffer sequences satisfy the convolution identities
\[
B_n(x+y)=\sum_{k=0}^n \binom{n}{k}x^kB_{n-k}(y),
\]
\[
H_n(x+y)=\sum_{k=0}^n \binom{n}{k}x^kH_{n-k}(y),
\]
and
\[
L^{\alpha}_n(x+y)=\sum_{k=0}^n L^{-1}_k(x)L^{\alpha}_{n-k}(y),
\]
respectively.  These polynomials also have extensions that are also Sheffer sequences, e.g., Bernoulli polynomials of higher order \cite{KKRD} and generalized Hermite polynomials \cite{S}.

If we renormalize a Sheffer sequence $s_n(x)$ and its associated sequence $p_n(x)$ by defining $\bar{s}_n(x)=s_n(x)/n!$ and $\bar{p}_n(x)=p_n(x)/n!$, then (\ref{eq:sheffer-sequence}) and (\ref{eq:polynomial-binomial-type}) are equivalent to
\begin{equation} \label{eq:sheffer-sequence-normalized}
\bar{s}_n(x+y)=\sum_{k=0}^n \bar{p}_k(x)\bar{s}_{n-k}(y)
\end{equation}
and
\begin{equation} \label{eq:polynomial-binomial-type-normalized}
\bar{p}_n(x+y)=\sum_{k=0}^n \bar{p}_k(x)\bar{p}_{n-k}(y),
\end{equation}
respectively.  
Identities (\ref{eq:sheffer-sequence-normalized}) and (\ref{eq:polynomial-binomial-type-normalized}) form the basis for our main result.

\begin{theorem} \label{th:polynomial-digital-binomial-theorem}
 Let $n$ be a non-negative integer with base $b$ expansion $n=n_{N-1}b^{N-1}+\cdots +n_0b^0$ and $\{s_n(x)\}$ be a Sheffer polynomial sequence with associated sequence $\{p_n(x)\}$.  Then
\begin{equation} \label{eq:digital-binomial-theorem-sheffer-sequence}
 \prod_{i=0}^{N-1}\bar{s}_{n_i}(x_i+y_i) = \sum_{0\leq m\preceq_b n}\left(\prod_{i=0}^{N-1}\bar{p}_{m_i}(x_i) \prod_{i=0}^{N-1}\bar{s}_{n_i-m_i}(y_i) \right)
\end{equation}
and
\begin{equation} \label{eq:digital-binomial-theorem-binomial-type}
 \prod_{i=0}^{N-1}\bar{p}_{n_i}(x_i+y_i) = \sum_{0\leq m\preceq_b n}\left(\prod_{i=0}^{N-1}\bar{p}_{m_i}(x_i) \prod_{i=0}^{N-1}\bar{p}_{n_i-m_i}(y_i) \right),
\end{equation}
where $m=m_{N-1}b^{N-1}+\cdots +m_0b^0$.
\end{theorem}

In the case where  $x_0=\ldots = x_{N-1}$ and $y_0=\ldots = y_{N-1}$, we obtain as a corollary the following result:

\begin{corollary} \label{co:polynomial-digital-binomial-theorem}
For $n=b^N-1$, we have
\begin{equation} \label{eq:polynomial-digital-binomial-theorem}
\bar{s}_{b-1}(x+y)^N= \sum_{0\leq m\leq n}\left(\prod_{i=0}^{N-1}\bar{p}_{m_i}(x) \prod_{i=0}^{N-1}\bar{s}_{b-1-m_i}(y) \right)
\end{equation}
and
\begin{equation}
\bar{p}_{b-1}(x+y)^N= \sum_{0\leq m\leq n}\left(\prod_{i=0}^{N-1}\bar{p}_{m_i}(x) \prod_{i=0}^{N-1}\bar{p}_{b-1-m_i}(y) \right)
\end{equation}
\end{corollary}

If we specialize to Bernoulli polynomials by setting $s_n(x)=B_n(x)$ and $p_n(x)=x^n$, then (\ref{eq:polynomial-digital-binomial-theorem}) gives the following result.
\begin{theorem} \label{th:bernoulli-digital-binomial-theorem}
For $n=b^N-1$, we have
\begin{equation} \label{eq:bernoulli-digital-binomial-theorem}
B_{b-1}(x+y)^N= \sum_{0\leq m\leq n}\left(x^{s_b(m)} \prod_{i=0}^{N-1}B_{b-1-m_i}(y) \right)
\end{equation}
where $s_b(m)$ is the base-$b$ sum-of-digits function.
\end{theorem}
Similar formulas can be obtained for other special polynomials such as Hermite and Laguerre polynomials.  Also, we remark that setting $y=0$ in (\ref{eq:bernoulli-digital-binomial-theorem}) yields a formula for higher powers of Bernoulli polynomials in terms of Bernoulli numbers $B_n:=B_n(0)$ and the sum-of-digits function:
\[
B_{b-1}(x)^N= \sum_{0\leq m\leq n}\left(x^{s_b(m)} \prod_{i=0}^{N-1}B_{b-1-m_i} \right).
\]

The proof of Theorem \ref{th:polynomial-digital-binomial-theorem} will be given in the next section where we investigate a Sheffer sequence analog of the Sierpi\'nski matrix and use its multiplicative property to derive (\ref{eq:digital-binomial-theorem-sheffer-sequence}) and (\ref{eq:digital-binomial-theorem-binomial-type}).

\section{Sierpinski Matrices of Sheffer Type}

Throughout this paper we assume that $b$ is an integer greater than 1.  We begin by introducing the notion of digital dominance as defined in \cite{BEJ} (see also \cite{N2}).

\begin{definition} \label{de:digital-dominance}
Let $m$ and $n$ be non-negative integers with base $b$ expansions $m=m_{N-1}b^{N-1}+\cdots +m_0b^0$ and $n=n_{N-1}b^{N-1}+\cdots +n_0b^0$, respectively.  We denote $m\preceq_b n$ to mean that $m$ is \textrm{digitally} less than $n$ in base $b$, i.e., $m_k\leq n_k$ for all $k=0,\ldots,N-1$.
\end{definition}

Next, we define a sequence of generalized Sierpi\'nski matrices corresponding to the Sheffer sequence $s_n(x)$ and its associated sequence $p_n(x)$.
\begin{definition} \label{de:1}
Let $N$ be a non-negative integer.  Denote $\mathbf{x}_N=(x_0,\ldots, x_{N-1})$.   If $N=0$, we set $S_{b,0}(\mathbf{x}_0)=P_{b,0}(\mathbf{x}_0)=1$.
For $N>0$, we define the $N$-variable Sierpi\'nski matrices 
\[
S_{b,N}(\mathbf{x}_N)=(\alpha_N(j,k,\mathbf{x}_N))
\]
matrices 
\[
P_{b,N}(\mathbf{x}_N)=(\beta_N(j,k,\mathbf{x}_N))
\]
 of dimension $b^N\times b^N$ by
\begin{equation} \label{eq:formula-for-entries-s(x)}
\alpha_N(j,k,\mathbf{x}_N)=\left\{
\begin{array}{cl}
\prod_{i=0}^{N-1}\bar{s}_{d_i}(x_i) & 
\begin{array}{l} 
\mathrm{if} \ 0\leq k\leq j \leq b^N-1  \\
\mathrm{and} \  k\preceq_b j; 
\end{array}  \\
\\
0, & 
\begin{array}{l}
\mathrm{otherwise},
\end{array}
\end{array}
\right.
\end{equation}
and
\begin{equation} \label{eq:formula-for-entries-arbitrary-p(x)}
\beta_N(j,k,\mathbf{x}_N)=\left\{
\begin{array}{cl}
\prod_{i=0}^{N-1}\bar{p}_{d_i}(x_i) & 
\begin{array}{l} 
\mathrm{if} \ 0\leq k\leq j \leq b^N-1  \\
\mathrm{and} \  k\preceq_b j; 
\end{array}  \\
\\
0, & 
\begin{array}{l}
\mathrm{otherwise},
\end{array}
\end{array}
\right.
\end{equation}
respectively, where $j-k=d_0b^0+d_1b^1+\ldots + d_{N-1}b^{N-1}$ is the base-$b$ expansion of $j-k$. 
\end{definition}

The following lemma gives a recurrence for $S_{b,N}(\mathbf{x}_N)$ and  $P_{b,N}(\mathbf{x}_N)$.

\begin{lemma} \label{le:multivariable}
The generalized Sierpi\'nski matrices $S_{b,N}(\mathbf{x}_N)$ and  $P_{b,N}(\mathbf{x}_N)$ satisfy the recurrence
\begin{equation} \label{eq:le:s(x)}
S_{b,N+1}(\mathbf{x}_{N+1})=S_{b,1}(x_N)\otimes S_{b,N}(\mathbf{x}_N),
\end{equation}
and
\begin{equation} \label{eq:le:p(x)}
P_{b,N+1}(\mathbf{x}_{N+1})=P_{b,1}(x_N)\otimes P_{b,N}(\mathbf{x}_N),
\end{equation}
respectively, where we define
\begin{align*}
S_{b,1}(x) & =\left(
\begin{array}{cccll}
1 & 0 & 0 & \cdots & 0 \\
\bar{s}_1(x) & 1 & 0 & \cdots & 0 \\
\bar{s}_2(x) & \bar{s}_1(x) & 1 & \cdots & 0 \\
\vdots & \vdots & \vdots & \ddots  & \vdots \\
\bar{s}_{b-1}(x) &\bar{s}_{b-2}(x) &\bar{s}_{b-3}(x) & \cdots &  1
\end{array}
\right)=
\left\{
\begin{array}{cl}
\bar{s}_{j-k}(x), & \textrm{if } 0\leq k \leq j \leq b-1; \\
0, & \mathrm{otherwise}
\end{array}
\right. \\
\end{align*}
and
\begin{align*}
P_{b,1}(x) & =\left(
\begin{array}{cccll}
1 & 0 & 0 & \cdots & 0 \\
\bar{p}_1(x) & 1 & 0 & \cdots & 0 \\
\bar{p}_2(x) & \bar{p}_1(x) & 1 & \cdots & 0 \\
\vdots & \vdots & \vdots & \ddots  & \vdots \\
\bar{p}_{b-1}(x) &\bar{p}_{b-2}(x) &\bar{p}_{b-3}(x) & \cdots &  1
\end{array}
\right)=
\left\{
\begin{array}{cl}
\bar{p}_{j-k}(x), & \textrm{if } 0\leq k \leq j \leq b-1; \\
0, & \mathrm{otherwise}.
\end{array}
\right. \\
\end{align*}
\end{lemma}

\begin{proof} Following \cite{MN}, we shall prove (\ref{eq:le:s(x)}) by induction on $N$.  
It is clear that (\ref{eq:le:s(x)}) holds for $N=0$.  Next, assume that (\ref{eq:le:s(x)}) holds for $N-1$.  To prove that(\ref{eq:le:s(x)}) holds for $N$, we express $S_{N+1}(\mathbf{x}_{N+1})$ a $b\times b$ matrix of blocks $(A_{p,q})_{0\leq p,q \leq b-1}$:
\begin{align*}
S_{N+1}(\mathbf{x}_{N+1})&=\left(\begin{array}{lll}A_{0,0}& \ldots & A_{0,b-1}\\ \vdots & \ddots & \vdots \\ A_{b-1,0}& \ldots & A_{b-1,b-1}\end{array}\right),
\end{align*} 
where each $A_{p,q}$ is a square matrix of size $b^N$. 
We consider two cases depending on the position of $A_{p,q}$:

\vskip 6pt
\noindent Case 1. $p<q$.  Then by definition of $S_{b,N+1}(\mathbf{x}_{N+1})$ we have that $\alpha_{N+1}(p,q,\mathbf{x}_{N+1})=0$, which implies $A_{p,q}=0$.

\vskip 6pt
\noindent Case 2. $p\geq q$.  
Let $\alpha_{N+1}(j,k,\mathbf{x}_{N+1})$ be an arbitrary entry of $A_{p,q}$.  Then $pb^N\leq j \leq (p+1)b^N-1$ and $qb^N \leq k \leq (q+1)b^N-1$.  Set $j'=j-pb^N$ and $k'=k-qb^N$.  
If $j<k$, then by definition $\alpha_{N+1}(j,k,\mathbf{x}_N)=0$.  Therefore, we assume $j\geq k$.  Let $j-k=d_0b^0+d_1b^1+\dots + d_{N}b^{N}$, where $d_N=p-q$.  Then $j'-k'=d_0b^0+d_1b^1+\dots +d_{N-1}b^{N-1}$.  
Since $k \preceq_b j$ if and only if $k' \preceq_b j'$,we have
\begin{align*}
\alpha_{N+1}(j,k,\mathbf{x}_{N+1}) & =\left\{
\begin{array}{cl}
\prod_{i=0}^{N}\bar{s}_{d_i}(x_i)  & \mathrm{if} \ 0\leq k\leq j \leq b^{N+1}-1 \ \mathrm{and} \  k \preceq_b j;  \\ 
\\
0 & \textrm{otherwise.}
\end{array}
\right. \\
& =\bar{s}_{d_N}(x_N)  \alpha_{b,N}(j',k',\mathbf{x}_N). \end{align*}
It follows that
\[
A_{p,q}=\bar{s}_{p-q}(x_N) S_{b,N}(\mathbf{x}_N,\mathbf{r}_N)
\]
and hence $S_{b,N+1}(\mathbf{x}_{N+1})=S_{b,1}(x_N,r_N)\otimes S_{b,N}(\mathbf{x}_N,\mathbf{r}_N)$.  This proves (\ref{eq:le:s(x)}).  The proof for (\ref{eq:le:p(x)}) is analogous and will be omitted.

\end{proof}

The generalized Sierpi\'nski matrices $S_{b,N}(\mathbf{x}_N)$ and  $P_{b,N}(\mathbf{x}_N)$ satisfy the following multiplicative property:
\begin{theorem} \label{th:multplicative} Let $N$ be a non-negative integer.  Then
\begin{equation} \label{eq:multplicative}
P_{b,N}(\mathbf{x}_N)S_{b,N}(\mathbf{y}_N)=S_{b,N}(\mathbf{x}_N + \mathbf{y}_N)
\end{equation}
and
\begin{equation} \label{eq:multplicative-p(x)}
P_{b,N}(\mathbf{x}_N)P_{b,N}(\mathbf{y}_N)=P_{b,N}(\mathbf{x}_N + \mathbf{y}_N),
\end{equation}
where we define
\[
\mathbf{x}_N + \mathbf{y}_N=(x_0+y_0,x_1+y_1,\ldots,x_{N-1}+y_{N-1}).
\]
\end{theorem}

\begin{proof}
We shall prove (\ref{eq:multplicative}) using induction.  To prove that (\ref{eq:multplicative}) holds for the base case $N=1$, let $\gamma(j,k)$ denote the $(j,k)$-entry of $T=P_{b,1}(\mathbf{x}_1)S_{b,1}(\mathbf{y}_1)$.  Since $T$ is lower-triangular, it follows that $\gamma(j,k)=0$ if $j<k$.  Therefore, we assume $j\geq k$.  By definition of $S_{b,1}(\mathbf{x}_1)$ and $P_{b,1}(\mathbf{y}_1)$, we have
\[
\gamma(j,k)=\sum_{i=k}^j \bar{p}_{j-i}(x_0)\bar{s}_{i-k}(y_0)=\sum_{i=0}^{j-k} \bar{p}_{i}(x_0)\bar{s}_{j-k-i}(y_0).
\]
Since $s_n(x)$ is a Sheffer sequence, it follows from (\ref{eq:sheffer-sequence}) that
\[
\gamma(j,k)=\bar{s}_{j-k}(x_0+y_0),
\]
or equivalently, 
\[
P_{b,1}(\mathbf{x}_1,\mathbf{r}_1)S_{b,1}(\mathbf{y}_1,\mathbf{r}_1)=
S_{b,1}(\mathbf{x}_1 + \mathbf{y}_1,\mathbf{r}_1).
\]
This proves (\ref{eq:multplicative}) for $N=1$.

Next, assume that (\ref{eq:multplicative}) holds for arbitrary $N$.  To prove that (\ref{eq:multplicative}) holds for $N+1$, we employ Lemma \ref{le:multivariable} and the mixed-property of a Kronecker product:
\begin{align*}
P_{b,N+1}(\mathbf{x}_{N+1})S_{b,N+1}(\mathbf{y}_{N+1}) &=(P_{b,1}(x_N)\otimes P_{b,N}(\mathbf{x}_N)(S_{b,1}(y_N)\otimes S_{b,N}(\mathbf{y}_N))\\
&=(P_{b,1}(x_N)S_{b,1}(y_N))\otimes (P_{b,N}(\mathbf{x}_N)S_{b,N}(\mathbf{y}_N)).
\end{align*}
Moreover, by the induction hypothesis and Lemma \ref{le:multivariable} again, we obtain 
\begin{align*}
P_{b,N+1}(\mathbf{x}_{N+1})S_{b,N+1}(\mathbf{y}_{N+1})&=S_{b,1}(x_N+y_N)\otimes S_{b,N}(\mathbf{x}_N+\mathbf{y}_N) \\
&= S_{b,N+1}(\mathbf{x}_{N+1}+\mathbf{y}_{N+1}).
\end{align*}
Hence, (\ref{eq:multplicative}) holds for $N+1$.  The proof of (\ref{eq:multplicative-p(x)}) is similar and will be omitted.
\end{proof}

\begin{proof}[Proof of Theorem \ref{th:polynomial-digital-binomial-theorem}]
We equate the matrix entries at position $(n,0)$ on both sides of (\ref{eq:multplicative})
to obtain
\begin{align*}
\alpha_N(n,0,\mathbf{x}_N+\mathbf{y}_N) & =\sum_{0\leq m\preceq_b n}\beta_N(n,m,\mathbf{x}_N)\alpha_N(m,0,\mathbf{y}_N) \\
& =\sum_{0\leq m\preceq_b n}\beta_N(n,n-m,\mathbf{x}_N)\alpha_N(n-m,0,\mathbf{y}_N).
\end{align*}
This yields (\ref{eq:digital-binomial-theorem-sheffer-sequence}) as desired.  The derivation of (\ref{eq:digital-binomial-theorem-binomial-type}) is similar and left for the reader to verify.
\end{proof}

We conclude with two ideas on how to extend Theorem \ref{th:polynomial-digital-binomial-theorem}.  The first is to iterate (\ref{eq:digital-binomial-theorem-sheffer-sequence}) and (\ref{eq:digital-binomial-theorem-binomial-type}) to obtain a trinomial version:

\begin{align*}
 \prod_{i=0}^{N-1}\bar{s}_{n_i}(x_i+y_i+z_i)  & = \sum_{0\leq m\preceq_b n}\left(\prod_{i=0}^{N-1}\bar{p}_{m_i}(x_i+y_i) \prod_{i=0}^{N-1}\bar{s}_{n_i-m_i}(z_i) \right) \\
 & = \sum_{0\leq m\preceq_b n}\sum_{0\leq l \preceq_b m}\left(\prod_{i=0}^{N-1}\bar{p}_{l_i}(x_i) \prod_{i=0}^{N-1}\bar{p}_{m_i-l_i}(y_i) \prod_{i=0}^{N-1}\bar{s}_{n_i-m_i}(z_i) \right)
\end{align*}
and similarly,
\begin{align*}
 \prod_{i=0}^{N-1}\bar{p}_{n_i}(x_i+y_i+z_i)  = \sum_{0\leq m\preceq_b n}\sum_{0\leq l \preceq_b m}\left(\prod_{i=0}^{N-1}\bar{p}_{l_i}(x_i) \prod_{i=0}^{N-1}\bar{p}_{m_i-l_i}(y_i) \prod_{i=0}^{N-1}\bar{p}_{n_i-m_i}(z_i) \right)
\end{align*}
where $l=l_{N-1}b^{N-1}+\cdots +l_0b^0$.
Of course, we can also iterate repeatedly to obtain to the multinomial formulas
\begin{align*}
& \prod_{i=0}^{N-1} \overline{s}_{n_i}( x^{(1)}_i+\cdots+x^{(d)}_i)  \\
& = \sum_{0\leq m^{(d-1)}\preceq_b n} \ldots \sum_{0\leq m^{(1)} \preceq_b m^{(2)}}\left(\prod_{i=0}^{N-1}\bar{p}_{m_i^{(1)}}(x^{(1)}_i) \cdots \prod_{i=0}^{N-1}\bar{p}_{m^{(d-1)}_i-m^{(d-2)}_i}(x^{(d-1)}_i) \prod_{i=0}^{N-1}\bar{s}_{n_i-m^{(d-1)}_i}(x^{(d)}_i) \right)
\end{align*}
and
\begin{align*}
& \prod_{i=0}^{N-1} \overline{p}_{n_i}( x^{(1)}_i+\cdots+x^{(d)}_i)  \\
& =   \sum_{0\leq m^{(d-1)}\preceq_b n} \ldots \sum_{0\leq m^{(1)} \preceq_b m^{(2)}}\left(\prod_{i=0}^{N-1}\bar{p}_{m_i^{(1)}}(x^{(1)}_i) \cdots \prod_{i=0}^{N-1}\bar{p}_{m^{(d-1)}_i-m^{(d-2)}_i}(x^{(d-1)}_i) \prod_{i=0}^{N-1}\bar{p}_{n_i-m^{(d-1)}_i}(x^{(d)}_i) \right)
\end{align*}
where $m^{(j)}=m^{(j)}_{N-1}b^{N-1}+\cdots +m^{(j)}_0b^0$.

The second idea is to view the binomial convolution identity for Sheffer sequences in the context of umbral calculus and linear functionals on the vector space of polynomials.  Define the product of two such linear functionals $L$ and $M$ by
\begin{equation}
\langle LM|x^n \rangle=\sum_{k=0}^n \binom{n}{k} \langle L|x^k\rangle \langle M | x^{n-k}\rangle.
\end{equation}
Then Roman and Rota \cite[Prop. 3.3, p. 102]{RR} proved that for any polynomial sequence $p_n(x)$ of binomial type, we have
\begin{equation} \label{eq:linear-functional}
\langle LM|p_n(x) \rangle=\sum_{k=0}^n \binom{n}{k} \langle L|p_k(x)\rangle \langle M | p_{n-k}(x) \rangle.
\end{equation}
The digital version of (\ref{eq:linear-functional}) and its multinomial generalization then becomes clear and extends \cite[Prop. 3.4, p. 103]{RR}; (see also \cite[Eq. (17), p. 874]{KK}):
\[
\langle LM|p_n(x) \rangle=\sum_{m=0}^n \left( \prod_{i=0}^{N-1}\langle L|p_{m_i}(x)\rangle \prod_{i=0}^{N-1}\langle M | p_{n_i-m_i}(x) \rangle \right)
\]
and
\begin{align*}
& \langle L_1\cdots L_d|p_n(x) \rangle \\
& = \sum_{0\leq m^{(d-1)}\preceq_b n} \ldots \sum_{0\leq m^{(1)} \preceq_b m^{(2)}} \left( \prod_{i=0}^{N-1}\langle L_1|p_{m_i^{(1)}}(x)\rangle \cdots  \prod_{i=0}^{N-1}\langle L_d | p_{m_i^{(d-1)}-m_i^{(d-2)}}(x) \rangle \prod_{i=0}^{N-1}\langle L_d | p_{n_i-m_i^{(d-1)}}(x) \rangle \right).
\end{align*}

\end{document}